\documentclass[12pt]{amsart}
\usepackage{amsmath,amssymb,amsfonts, amscd}
\usepackage{enumerate}
\usepackage{anysize}
\marginsize{3cm}{3cm}{3cm}{3cm}
\input xy
\xyoption{all}
\usepackage{pb-diagram}
\usepackage[all]{xy}
\input xy
\xyoption{all}

\theoremstyle{plain}
\newtheorem{thm}{Theorem}

\newtheorem{lemma}{Lemma}
\newtheorem{cor}{Corollary}

\newtheorem{prop}{Proposition}

\theoremstyle{definition} \theoremstyle{definition}

\theoremstyle{remark}

\newcommand{\GG}{\mathbb{G}}

\newcommand{\Q}{\mathbb{Q}}

\newcommand{\wT}{\widehat{T}}
\newcommand{\wG}{\widehat{G}}
\newcommand{\wQ}{\widehat{Q}}
\newcommand{\wpi}{\widehat{\pi}}
\newcommand{\wchi}{\widehat{\chi}}

\newcommand{\wrho}{\widehat{\rho}}

\newcommand{\Z}{\mathbb{Z}}

\newcommand{\C}{\mathbb{C}}

\newcommand{\Hom}{{\rm Hom}\,}

\def\SL{{\rm SL}}

\def\Tor{{\rm Tor}}

\def\GL{{\rm GL}}
\def\PGL{{\rm PGL}}

\def\SO{{\rm SO}}

\begin{document}

\title[Half the sum of positive roots, Coxeter and  Kostant]
{Half the sum of positive roots, the Coxeter element, and a theorem of Kostant}

\author{Dipendra Prasad}

\address{School of Mathematics, Tata Institute of Fundamental
Research, Colaba, Mumbai-400005, INDIA}
\email{dprasad@math.tifr.res.in}
\maketitle

\begin{abstract}
Interchanging character and co-character groups of a torus $T$ over a field $k$ introduces 
a contravariant functor $T \rightarrow \wT$. 
Interpreting $\rho:T\rightarrow \C^\times$, half the sum of positive roots for
$T$  a maximal torus in a simply connected semi-simple group $G$ (over $\C$) using this duality, we get a co-character
$\wrho: \C^\times \rightarrow \wT$  whose value at $e^{\frac{2 \pi i }{h}}$ ($h$ the Coxeter number)
is the Coxeter conjugacy class
of the dual group $\wG$. 
This point of view gives a rather transparant proof of a 
 theorem of Kostant on the character values of irreducible 
finite dimensional representations of $G$ at the Coxeter element: the proof amounting to 
the fact that in $\wG_{sc}$, the simply connected cover of $\wG$, there is a unique 
regular conjugacy class whose image in $\wG$ has order $h$ (which is the Coxeter conjugacy class).

\end{abstract}

\vspace{4mm}

The theorem of Kostant in the title of this paper is Theorem 2 in [Ko], and  deals with the character of a finite dimensional representation of a  semi-simple algebraic group $G$ over $\C$ at a very special element, the Coxeter element. Recall that one usually defines a Coxeter element --- or, rather a conjugacy class --- 
in a Weyl group (as a product of simple reflections), in this case in $N(T)/T$, where $T$ is a maximal torus in $G$, with $N(T)$ its normalizer in $G$.
The first observation is that if we lift this conjugacy class in $N(T)/T$ to $N(T)$ arbitrarily, we get a well-defined conjugacy class in $G$; we will denote this conjugacy class in $G$ to be $c(G)$, and call it the Coxeter conjugacy class in $G$. 
The conjugacy class $c(G)$ in $G$ has the distinguishing property that it is the unique regular 
semi-simple conjugacy class in $G$ (i.e., the connected component of its  centralizer in $G$ is a torus) such that its image in $G/Z$ is 
of the smallest order in the adjoint group $G/Z$. Throughout the paper we will assume without loss of generality
that $G$ is a semi-simple simply connected group over $\C$ with $T$ a fixed 
maximal torus in $G$.  (Since $c(G)$ and $zc(G)$ are conjugate in $G$ for any central element $z$ of $G$, any irreducible representation of $G$ 
with nonzero character value on $c(G)$ must have trivial central character. However, character values 
in this paper are calculated via Weyl character
formula which involves the character $\rho$, half the sum of positive roots, and this character is nontrivial on $Z$, so it does not help us to assume $G$ adjoint.)

Here is the theorem of Kostant for which we will offer another proof which will also give a more precise information on these character
values (which is also contained in Kostant's proof although he does not spell it out).

\begin{thm} Let $G$ be a semi-simple simply connected  group over $\C$, and $\pi$ a finite dimensional irreducible representation of $G$. 
Then the character $\Theta_\pi$ of $\pi$ at the element $c(G)$ takes one of the values $1,0,-1$.
\end{thm}

We recall that to a reductive algebraic group $G$ over $\C$, there is associated 
the dual group 
$\widehat{G}$ which is a reductive algebraic group over $\C$ with root datum which is the dual to that of $G$. 
Fix  a maximal torus $T$ in $G$, and 
a maximal 
torus $\widehat{T}$ in $\widehat{G}$, such 
that  there is a canonical isomorphism between the character
group of $T$ and the co-character group of $\widehat{T}$, and as a result we have the identifications
$$\widehat{T}(\C) = {\rm Hom} [\C^\times, \widehat{T}] \otimes_\Z  \C^\times  = 
{\rm Hom}[T,\C^\times]\otimes_\Z \C^\times,$$ 
all the homomorphisms being algebraic. Thus given a 
character $\chi: T \rightarrow \C^\times$, it gives rise to a co-character,
 $$\widehat{\chi}: \C^\times \longrightarrow  \widehat{T}(\C)= {\rm Hom}[T,\C^\times]\otimes_\Z \C^\times, $$ 
 given by $z \longrightarrow \chi \otimes z$.

For  a semi-simple simply connected algebraic group $G$ over $\C$, let $\rho$ be  half
the sum of positive roots of a maximal torus $T$ in $G$ (for any fixed choice of positive roots). 
It is clear from the definition of $\rho$ that 
the pair $(T, \rho)$ is well-defined up to
conjugacy by $G(\C)$; in particular, the restriction of the character $\rho$ to $Z$, the center of $G(\C)$, is a well defined
character of $Z$ which is of order $\leq 2$, to be denoted by $\rho_{Z}: Z \rightarrow \Z/2$. 
By (Pontrajagin) duality, we get 
a homomorphism $\rho^\vee_{Z}: \Z/2 \rightarrow Z^\vee$ 
where $Z^\vee$ denotes the character group of $Z$. 

Let $\wG_{sc}$ be the simply connected cover  of $\wG$ whose center can be identified to
$Z^\vee$, the character group of $Z$. We shall see later that the image of the nontrivial element 
in $\Z/2$ under the  homomorphism $\rho^\vee_{Z}: \Z/2 \rightarrow Z^\vee$ gives an important element 
in the center of $\wG_{sc}$ which determines whether an irreducible selfdual representation of $\wG_{sc}$ is 
orthogonal or symplectic.

Let $T'=T/Z$ be the maximal torus in the adjoint group $G/Z$, sitting in the following exact sequence:
$$0 \longrightarrow Z \longrightarrow T(\C) \longrightarrow T'(\C)\longrightarrow 0.$$
It is easy to see that any homomorphism from $Z$ to $\C^\times$ can be extended to an algebraic character of $T$, therefore we have an exact 
sequence:
$$0 \longrightarrow \Hom[T'(\C),\C^\times]  \longrightarrow \Hom[T(\C), \C^\times] \longrightarrow \Hom[Z,\C^\times]\longrightarrow 0.$$

Tensoring this exact sequence by  $\C^\times$, we get an exact sequence:
$$0 \longrightarrow \Tor^1[Z^\vee, \C^\times]  \longrightarrow \Hom[T'(\C), \C^\times] \otimes \C^\times \longrightarrow 
\Hom[T(\C),\C^\times] \otimes\C^\times \longrightarrow 0.$$
Since for any finite abelian group $A$, $\Tor^1[A,\C^\times] \cong A$, this gives,
$$0 \longrightarrow Z^\vee  \longrightarrow \widehat{T'}(\C) \longrightarrow  \wT(\C) 
\longrightarrow 0.$$

Similarly we have the exact sequence:
$$0 \longrightarrow \Z/2  \longrightarrow \C^\times \stackrel{2}\longrightarrow \C^\times 
\longrightarrow 0.$$

Now consider the following pair of short exact sequences,

 $$
\begin{CD}
0 @>>>   Z @>>>  T(\C) @>>>   T'(\C)  @>>> 0\\ 
& & @V \rho|_{Z} VV @V \rho VV @V 2\rho VV \\
0 @>>> \Z/2 @>>> 
 \C^\times @> 2 >>  \C^\times  @>>> 0.
\end{CD}
$$
Applying the functor $\Hom[-,\C^\times]$ followed by $\otimes \C^\times$ as before, we get the following commutative diagram of exact 
sequences, 

 $$
\begin{CD}
0 @>>>   
Z^\vee @>>>  \widehat{T'}(\C) @>>>   \wT(\C)  @>>> 0\\ 
& & @A AA  @A 2\wrho AA  @A \wrho AA \\
0 @>>> 
\Z/2  @>>> 
 \C^\times @> 2>>  \C^\times  @>>> 0.
\end{CD}
$$
\vspace{4mm}

which is part of the following larger diagram:

 $$
\begin{CD}
0 @>>>   
Z^\vee @>>>  \wG_{sc}(\C) @>>>   \wG(\C)  @>>> 0\\ 
& & @A AA  @AAA  @A  AA \\
0 @>>>   
Z^\vee @>>>  \widehat{T'}(\C) @>>>   \wT(\C)  @>>> 0\\ 
& & @A AA  @AAA  @A  AA \\
0 @>>> 
\Z/2 @>>> 
 \C^\times @> 2>>  \C^\times  @>>> 0.
\end{CD}
$$
\vspace{4mm}

We simplify the above diagram to:
 $$
\begin{CD}
0 @>>>   Z^\vee @>>>  \wG_{sc}(\C) @>>>   \wG(\C)  @>>> 0\\ 
& & @A AA  @AAA  @A  AA \\
0 @>>> 
\Z/2 @>>> 
 \C^\times @> 2>>  \C^\times  @>>> 0,
\end{CD}
$$
\vspace{4mm}
and then modify to:

$$
\begin{CD}
0 @>>>   Z^\vee @>>>  \wG_{sc}(\C) @>>>   \wG(\C)  @>>> 0\\ 
& & @A \wrho_Z AA  @A \tilde{\phi} AA  @A  \phi AA \\
0 @>>> 
\Z/2 @>>> 
 \SL_2(\C) @> >>  \PGL_2(\C)  @>>> 0,\\
& & @A AA  @AAA  @A  AA \\
0 @>>> 
\Z/2 @>>> 
 \C^\times @> 2>>  \C^\times  @>>> 0.
\end{CD}
$$
\vspace{4mm}

The main point about this commutative diagram is the fact that the homomorphisms $\phi$ and $\tilde{\phi}$ 
correspond (via Jacobson-Morozov theorem) to regular unipotent conjugacy classes in $\wG(\C)$ and $\wG_{sc}(\C)$ respectively.

\vspace{4mm}

\begin{prop} 
The homomorphisms $\phi$ and $\tilde{\phi}$ 
correspond to regular unipotent classes in $\wG(\C)$ and $\wG_{sc}(\C)$ 
respectively, and the element 
$\phi( e^{2\pi i /h}) \in  \wG(\C)$ 
represents  the Coxeter conjugacy class in the adjoint group
$\wG(\C)$.

\end{prop} 

\begin{proof}Let $\tau: \GG_m(\C) \rightarrow \wT(\C)$ be the cocharacter whose action on all simple root spaces 
of $\widehat{B} < \wG$ is via the identity  maps $\GG_m(\C) \rightarrow \GG_m(\C)$. This cocharacter $\tau$ is part of the 
principal $\PGL_2(\C)$ inside $\wG(\C)$, and it is part of Kostant's work in [Ko2] that  
$\tau( e^{2\pi i /n}) 
\in  G^\vee(\C)$ is the Coxeter conjugacy class in the adjoint group
$\wG(\C)$. It suffices then to observe that $\tau = \rho^\vee$ 
which is nothing but the well-known identity:
$$ 2 \langle \rho, \alpha \rangle  =  \langle \alpha,\alpha \rangle $$
for all $+ve$ simple roots $\alpha$ in $G_{sc}$ which is the same as,
$$ \langle \rho, \alpha^\vee \rangle  =  1,$$
i.e.,
$$ \rho \circ  \alpha^\vee = 1: \GG_m(\C) \stackrel{\alpha^\vee}\longrightarrow T(\C) \stackrel{\rho} \longrightarrow\GG_m(\C),$$ 
which by applying duality is the same as:

$$ \widehat{\alpha}^\vee \circ \wrho = 1: \GG_m(\C) \stackrel{\wrho}\longrightarrow \wT(\C) \stackrel{\widehat{\alpha}^\vee} \longrightarrow \GG_m(\C),$$
proving that $\tau = \rho^\vee$. \end{proof}

\begin{cor}  The image of the homomorphism $\rho_{Z}^\vee: \Z/2 \rightarrow Z^\vee$ 
defines an element $z_0$ of order 1 or 2 in the
center of $\wG_{sc}$ which acts on an irreducible  self-dual representation $\pi$ of $\wG_{sc}$ by 1 if and only the 
representation $\pi$ carries a nondegenerate symmetric bilinear form.
\end{cor}

\begin{proof}
This is a well-known consequence of the fact that the homomorphism  $\tilde{\phi}: \SL_2(\C) \rightarrow \wG_{sc}(\C)$ 
corresponds to the regular unipotent class in $\wG_{sc}(\C)$.
\end{proof}

\vspace{4mm}

\begin{cor} The character $\rho:T\rightarrow \C^\times$ is trivial on $Z(G)$ if and only if $z_0=1$, which is if and only if the Coxeter conjugacy class of the adjoint group $\wG$ lifts to an element of the simply 
connected group $\wG_{sc}$ of the same oder.
\end{cor}

\vspace{2mm}

\vspace{4mm}

{\bf Example:} If $G= Sp_{2n}$ with roots $\{\pm e_i \pm e_j, \pm 2e_k, {\rm ~~~ for ~~~} 1 \leq i,j,k \leq n \}$, then 
$\{e_i \pm e_j, i< j, {\rm ~~~and ~~~} 2e_k\}$ can be 
taken for a set of positive roots, and hence $\rho = ne_1+\cdots + e_n$. Therefore $\rho(-1) = (-1)^{n(n+1)/2}$. Therefore $\rho(-1) = 1$ if 
$n \equiv 0,-1 \bmod 4$, and 
$\rho(-1) = -1$ if $n  \equiv 1,2 \bmod 4$.

On the other hand, an irreducible representation of the Spin group ${\rm Spin}_{2n+1}$ 
is always self-dual, and is always orthogonal if 
$n \equiv 0,-1 \bmod 4$, and for $n \equiv 1,2 \bmod 4$ it is orthogonal if and only if its central character is trivial, i.e., if it is a representation of 
$SO_{2n+1}$; these are well-known conclusions for the spinor representation of ${\rm Spin}_{2n+1}$.

Since $\rho = ne_1+\cdots + e_n$, we note that 
$$\wrho: \C^\times \rightarrow \wG = \SO_{2n+1}(\C),$$ is given by,

$$\left  ( \begin{array}{ccccccccc}
t^n& {} & {} & {}  &  {}   & {}  &{} & &  \\
{}& t^{n-1}  & {} & {}   &  {}  &  {}  & {} & &  \\
{}& {}  & \cdot  & {}  &   {} & {}   & {} &  &\\
{}& {}  &   & {t}  &   {} & {}   &{} &  &\\
{}& {}  &   & {}  &   {1} & {}   &{} &  &\\
{} &  {}  &   {}   &  &  & t^{-1} &  &   & \\
{}& {}  &   & {}  &   {} &   {\cdot} &    & &    \\
{} & {}   &  {}  & {}  &  & & t^{-(n-1)} & &  \\
{} & {} & {}  & {}  & {} & & & t^{-n}&
\end{array} \right)\cdot
$$

\vspace{2mm}

 This is clearly the restriction of the representation $(\det^{-n}){\rm Sym}^{2n} (\C \oplus \C)$ of $\PGL_2(\C)$
restricted to the diagonal matrix:$\left  ( \begin{array}{cc}  t & 0 \\ 0 & 1 \end{array} \right)
$, and $\wrho( e^{2\pi i /{2n}}) \in \SO_{2n+1}(\C)$ is the Coxeter conjugacy class. (Recall that the Coxeter number
for $\SO_{2n+1}(\C)$ is $2n$.)

\section{A lemma on torsion points of isogenous tori}
Let $$0 \longrightarrow Z \longrightarrow T \stackrel{\pi}\longrightarrow Q\longrightarrow 0,$$ 
be an exact sequence of groups over $\C$ with $T$ and $Q$ tori of the same dimension, with the dual exact sequence,
$$0 \longrightarrow Z^\vee \longrightarrow \wQ \stackrel{\wpi}\longrightarrow  \wT \longrightarrow 0.$$ 

For any integer $n \geq 1$, let 
$$T_n = \{ t \in T(\C) | n \pi(t) = 0 \in Q(\C)\}.$$   
Similarly define,
$$\wQ_n = \{ t \in \wQ(\C) | n \wpi(t) = 0 \in \wT(\C)\}.$$

\begin{lemma} The character group of $T_n$ is isomorphic to $\wQ_n$ via an isomorphism which 
is  equivariant under those automorphisms of $T$ which preserve $Z$. \end{lemma}
\begin{proof}
It is easy to see that the character group $T_n^\vee$ of $T_n$ is,

\begin{eqnarray}T_n^\vee &\cong & \frac{X^{*}(T)}{nX^{*}(Q)} 
\end{eqnarray}

Note that if there is an exact sequence,
$$0 \longrightarrow X_*(T_1) \longrightarrow X_*(T_2) \longrightarrow A\longrightarrow 0,$$ 
then tensoring this exact sequence by $\C^\times$, we get
$$0 \longrightarrow {\rm Tor}^1[A,\C^\times] \longrightarrow X_*(T_1)\otimes \C^\times \longrightarrow X_*(T_2) \otimes \C^\times \longrightarrow 0.$$ Identifying ${\rm Tor}^1[A,\C^\times] $ to $A$ and $X_*(T) \otimes \C^\times $ to $T(\C)$, we get an exact sequence,
$$0 \longrightarrow A \longrightarrow T_1(\C) \longrightarrow T_2(\C) \longrightarrow 0.$$ 

Using these considerations,  it follows that 
$$T_n \cong \frac{X_{*}(Q)}{nX_{*}(T)}.$$
Applying this conclusion to the exact sequence,
$$0 \longrightarrow Z^\vee \longrightarrow \wQ \stackrel{\wpi}\longrightarrow  \wT \longrightarrow 0,$$ 

\begin{eqnarray}\wQ_n &\cong& \frac{X_{*}(\wT)}{nX_{*}(\wQ)}\cong \frac{X^{*}(T)}{nX^{*}(Q)}.
\end{eqnarray}

Identities $(1)$ and $(2)$ prove the Lemma.\end{proof}

\section{Proof of Kostant's theorem}

\begin{thm}Let $G$ be a semi-simple simply connected group over $\C$ with $\widehat{G}$
its dual group with a fixed maximal torus $T$ in $G$ and $\wT$ in $\widehat{G}$. Let $c(G)$ be the Coxeter conjugacy class in $G(\C)$ and
 $h$ the Coxeter number. Let $\pi_\lambda$ be a finite dimensional representation of 
$G(\C)$ with highest weight $\lambda: T \rightarrow \C^\times$.
Then the character of $\pi_\lambda$
at $c(G)$ is either $1,0,-1$, and it is nonzero if and only if 
$\widehat{\lambda\cdot \rho }(e^{2\pi i /h})$ 
is conjugate to $\widehat{\rho}(e^{2\pi i /h})$, 
both of them belonging to the Coxeter conjugacy class in $\wG$. 
If the character value is nonzero, then $\lambda$ must be trivial on the center of $G$, and therefore 
$\widehat{\lambda }(e^{2\pi i /h})$ makes sense as an element of $\widehat{T/Z}(\C)$. 
Lifting $\wrho (e^{2\pi i /h})$ from $\wT(\C)$  arbitrarily to $\widehat{T/Z}(\C)$, we have 
two elements $\widehat{\lambda\cdot \rho }(e^{2\pi i /h})$ and  $\widehat{\rho}(e^{2\pi i /h})$ in 
$\widehat{T/Z}(\C)$ (a maximal torus in the simply connected group $\wG_{sc}$) which are then  conjugate
to each other by a unique element $w_0$ of the Weyl group.  The sign of the
character of $\pi_\lambda$ at $c(G)$ is the sign of this element of the Weyl group.
\end{thm}

\begin{proof} The proof of this theorem is a direct consequence of the Weyl character 
formula which we recall gives the character of an irreducible representation of $G$ with highest weight $\lambda$ as a quotient of what's called the Weyl numerator 
by the Weyl denominator. Let $\rho$ be half the sum of positive roots. Then the 
Weyl numerator at an element $t$ of the  maximal torus $T$ is: 
$$\sum_w (-1)^{\ell(w)}(\lambda \cdot \rho )^w(t);$$ 
the Weyl denominator is the same as the Weyl numerator but for $\lambda = 1$. If 
$t$ is regular element of $T$ (such as the Coxeter element) then the Weyl denominator is non-zero, and therefore
the character at a regular element is zero if and only if the Weyl numerator is zero.

The character $\lambda\cdot \rho$ when restricted to $T_h = \{t \in T| h \pi(t) = 0 \}$ for $\pi: T \rightarrow T/Z= T_{ad}$ gives rise to a character of $T_h$ which can then be identified by Lemma 1 
to an element of $T_{ad, h}^\vee = \{t \in \widehat{T_{ad}}| h \pi^\vee(t) = 0\}$. If this element of $T_{ad, h}^\vee$ 
(recall that $\widehat{T}_{ad}$ is a maximal torus in the simply connected group $\wG_{sc}$)
turns out to be a singular element, 
 i.e., 
fixed by a non-trivial reflection $s_{\alpha}$, then for $W_0=\langle s_\alpha\rangle$,
$$(\lambda\cdot \rho)^{s_\alpha}(t) =  (\lambda\cdot \rho)(t),$$
for all elements $t \in T_h$, in particular for the Coxeter element $c(G)$ in $G$.   It follows that, 
$$\sum_{w \in W} (-1)^{\ell(w)}(\lambda \cdot \rho )^w(t) = \sum_{w\in W/W_0} \sum_{w'\in W_0} (-1)^{\ell(ww')}(\lambda \cdot \rho )^{ww'}(t) =0.$$ 
 
On the other hand, if the character $\lambda\cdot \rho$ when restricted to $T_h$ gives rise to a 
character of $T_h$ which when  identified
to an element of $T_{ad,h}^\vee$ is a regular element of $\widehat{T_{ad}}$, then there being a unique regular element 
in $T_{ad,h}^\vee$ up to conjugacy by $W$ which is in fact represented by the character $\rho$ on $T_h$ considered as an element of 
$T_{ad,h}^\vee$,  the character on $T_h$ defined by $(\lambda\cdot \rho)$ is the same as that defined by $\rho^{w_0}$ for a unique element 
$w_0$ of $W$. Therefore,
$$\sum_w (-1)^{\ell(w)}(\lambda \cdot \rho )^w(c(G)) = (-1)^{w_0}\sum_{W}  (-1)^{\ell(w)}(\rho )^{w}(c(G)) .$$ 
Therefore,
$$\Theta_{\pi_\lambda}(c(G))= (-1)^{w_0}.$$\end{proof}

\vspace{1cm}

{\bf Remark :} Given that the dual group does not really make its presence in theory of finite dimensional
representations of compact Lie groups, it may be worth emphasizing that the results of this paper ---all about
finite dimensional representations of compact Lie groups--- need the dual group in an essential way.

\vspace{2cm}

{\bf \Large Bibliography}

\vspace{1cm}

[Ko] Kostant, Bertram:
{\it On Macdonald's $\eta$-function formula, the Laplacian and generalized exponents.}
Advances in Math. 20 (1976), no. 2, 179--212. 

\vspace{3mm}

[Ko2] Kostant, Bertram: {\it The Principal Three Dimensional Subgroup and the Betti Numbers of a Complex Simple Lie Group}, American J. of 
Math. 81 (1959), 973-1032.

\end{document}